\documentclass[leqno]{amsart}

\usepackage{stmaryrd,graphicx}
\usepackage{amssymb,mathrsfs,amsmath,amscd,amsthm}
\usepackage{float}
\usepackage{graphicx,float}
\usepackage[all,cmtip]{xy}
\usepackage{amsfonts,txfonts,pxfonts,latexsym,wasysym}

\DeclareMathAlphabet{\mathpzc}{OT1}{pzc}{m}{it}

\newcommand{\spaces}{\mathbf{Top}}

\newcommand{\piz}{\pi_0}
\newcommand{\pizx}{\pi_0(X)}

\newtheorem{theorem}{Theorem}
\newtheorem{lemma}[theorem]{Lemma}
\newtheorem{proposition}[theorem]{Proposition}
\newtheorem{corollary}[theorem]{Corollary}
\theoremstyle{definition}\newtheorem{definition}[theorem]{Definition}

\theoremstyle{definition}\newtheorem{example}[theorem]{Example}
\theoremstyle{definition}
\theoremstyle{definition}\newtheorem{problem}[theorem]{Problem}
\theoremstyle{definition}\newtheorem{remark}[theorem]{Remark}
\numberwithin{theorem}{section}
\begin{document}

\title[Path-component spaces]{Realizing spaces as path-component spaces}
\author{Taras Banakh, Jeremy Brazas}
\begin{abstract}
The path component space of a topological space $X$ is the quotient space $\pi_0(X)$ whose points are the path components of $X$. We show that every Tychonoff space $X$ is the path-component space of a Tychonoff space $Y$ of weight $w(Y)=w(X)$ such that the natural quotient map $Y\to \pi_0(Y)=X$ is a perfect map. Hence, many topological properties of $X$ transfer to $Y$. We apply this result to construct a compact space $X\subset \mathbb{R}^3$ for which the fundamental group $\pi_1(X,x_0)$ is an uncountable, cosmic, $k_{\omega}$-topological group but for which the canonical homomorphism $\psi:\pi_1(X,x_0)\to \check{\pi}_1(X,x_0)$ to the first shape homotopy group is trivial.
\end{abstract}
\maketitle
\section{Introduction}
The \textit{path-component space} of a topological space $X$ is the set $\pi_0(X)$ of path components of $X$ equipped with the natural quotient topology inherited from $X$, i.e the natural map $X\to \pi_0(X)$ identifying each path component to a point is a quotient map. The path-component space of the space $\Omega(X,x_0)$ of based loops $S^1\to X$ is the usual fundamental group $\pi_1(X,x_0)$ equipped with a functorial topology that gives it the structure of a quasitopological group. Recent interest and applications of homotopy groups enriched with such topologies has brought more relevance to path-component spaces.

It is a beautiful and surprising result of Douglas Harris that every topological space is a path-component space \cite{Har80}. In particular, for every space $X$, Harris constructed a paracompact Hausdorff space $H(X)$ and a natural homeomorphism $\pi_0(H(X))\cong X$. Harris' result plays a key role in the application of the topological fundamental group to prove that every topological group is isomorphic to the topological fundamental group of some space \cite{Braztopgrp}. In particular, the free Markov topological group $F_M(X)$ on a space $X$ is isomorphic to the topological fundamental group of the reduced suspension $\Sigma H(X)_{+}$ of $H(X)_+=H(X)\cup \{\ast\}$ with an isolated basepoint. Harris' result is used in a similar fashion to prove a topological group analogue of the Nielsen-Schreier Theorem \cite{Br12NS}.

A remaining problem of relevance to topological fundamental groups is: given a particular class of spaces $\mathcal{C}$, identify a subclass $\mathcal{D}$ such that for every $Y\in\mathcal{C}$, we have $Y\cong \pi_0(X)$ for some $X\in\mathcal{D}$.

For instance, Banakh, Vovk, and W\' ojcik  show in \cite{BVW} that every first countable space $X$ is the path-component space of a complete metric space $\varoast X$ called the \textit{cobweb} of $X$. The construction of $\varoast X$ is actually quite similar to that of $H(X)$, however, it is not necessarily compact or separable when $X$ is. These results suggest that analogous constructions might be possible for other classes of spaces. In the current paper, we prove the following Theorem using techniques quite different from those in \cite{Har80} and \cite{BVW}.

\begin{theorem}\label{mainthm}
Every Tychonoff space $X$ is the path-component space of a Tychonoff space $Y$ of weight $w(Y)=w(X)$ such that the quotient map $q_Y:Y\to \pi_0(Y)=X$ is perfect, i.e. closed with compact preimages of points.
\end{theorem}
A topological property $\mathscr{P}$ is \textit{inversely preserved by perfect maps} if whenever $f:Y\to X$ is a perfect map and $X$ has property $\mathscr{P}$, then so does the preimage $f^{-1}(X)$. Such properties are numerous and include compactness, paracompactness, metrizability, Lindel\"of number, \v Cech completeness, Borel type, etc. We refer to \cite[Section 3.7]{Eng89} for more on perfect maps as well as the column of inverse invariants of perfect maps in the table on pg. 510 of \cite{Eng89}. Theorem \ref{mainthm} implies that any topological property of $X$ that is inversely preserved by perfect maps is also a property enjoyed by the space $Y$ whose path-component space is $X$.

\begin{corollary}
Every Tychonoff space $X$ is the path-component space of a Tychonoff space $Y$ such that $Y$ shares any topological property of $X$ that is inversely preserved by perfect maps.
\end{corollary}

As an application, we apply Theorem \ref{mainthm} to identify an interesting phenomenon in topologized fundamental groups. The path-component space of a loop space $\Omega(X,x_0)$ is the fundamental group $\pi_1(X,x_0)$ equipped with its natural quotient topology. In general, $\pi_1(X,x_0)$ need not be a topological group \cite{BFqtop}, however, it is a quasitopological group in the sense that inversion is continuous and multiplication is continuous in each variable \cite{Br10.1}. We refer to \cite{AT08} for basic theory of quasitopological groups. 

Since multiplication in $\pi_1(X,x_0)$ need not be continuous, one cannot take separation axioms for granted. While $T_0$ $\Leftrightarrow$ $T_1$ holds for all quasitopological groups it need not be the case that $T_1$ implies $T_2$ or that $T_2$ implies higher separation axioms. The relevance of separation axioms in fundamental groups was noted in \cite{BFqtop} where it is shown that if $\pi_1(X,x_0)$ is $T_1$, then $X$ has the homotopically path Hausdorff property introduced in \cite{FRVZ11}. The converse holds if $X$ is locally path connected. Furthermore, if $\pi_1(X,x_0)$ is $T_1$, then $X$ admits a generalized universal covering in the sense of \cite{FZ07}. 

There are several open questions related to the continuity of multiplication and separation axioms in quasitopological fundamental groups. We note here the open problems from \cite[Problem 36]{BFqtop}.

\begin{problem}
Must $\pi_1(X,x_0)$ be normal for every compact metric space $X$? 
\end{problem}

\begin{problem}
Is there a Peano continuum (connected, locally path-connected, compact metric space) $X$ for which $\pi_1(X,x_0)$ is $T_1$ but not $T_4$?
\end{problem}

To verify that $\pi_1(X,x_0)$ is Hausdorff, it is sufficient to know that $X$ is $\pi_1$\textit{-shape injective}, meaning the canonical homomorphism $\psi:\pi_1(X,x_0)\to \check{\pi}_1(X,x_0)$ to the first shape homotopy group is injective. Motivated by a construction in \cite{Br10.1}, we construct what is apparently the first example of a compact metric space for which $\pi_1(X,x_0)$ has been verified to be Hausdorff but for which $\psi$ is not injective. In doing so, we show in a strong way that $\pi_1$-shape injectivity is not a necessary condition for the $T_4$ separation axiom in fundamental groups of compact metric spaces.

\begin{theorem}\label{mainthm2}
There exists a compact metric space $X\subset \mathbb{R}^3$ such that $\pi_1(X,x_0)$ is an uncountable, cosmic, $k_{\omega}$-topological group such that $\psi:\pi_1(X,x_0)\to \check{\pi}_1(X,x_0)$ is the trivial homomorphism. In particular, $\pi_1(X,x_0)$ is $T_4$.
\end{theorem}
\section{Path-component spaces}
A \textit{path} in a topological space $X$ is a continuous function $\alpha:[0,1]\to X$ from the closed unit interval $[0,1]$. The \textit{path component} of a point $x\in X$ is set \[[x]=\{y\in X\text{ }|\text{ }\exists \text{ a path }\alpha:[0,1]\to X\text{ with }\alpha(0)=x\text{ and }\alpha(1)=y\}.\]The relation $[x]=[y]$, $x,y\in X$ is an equivalence relation on $X$; an equivalence class is called a \textit{path component} of $X$.

\begin{definition}
The \textit{path-component space} $\pi_0(X)$ of a topological space $X$, is the set $\pi_0(X)=\{[x]\text{ }|\text{ }x\in X\}$ of path components of $X$ equipped with the quotient topology with respect to the map $q_{X}:X\to \pi_0(X)$, $q(x)=[x]$ that collapses each path component to a point.
\end{definition}

If $f:X\to Y$ is a map, then $f([x])\subseteq [f(x)]$. Thus $f$ determines a well-defined map $f_{0}:\piz(X)\to \piz(Y)$ given by $f_0([x])=[f(x)]$. Moreover, $f_0$ is continuous since $f_{0}\circ q_{X}=q_{Y}\circ f$ and $q_X$ is quotient. Altogether, we obtain an endofunctor of the usual category $\spaces$ of topological spaces.
\begin{proposition} 
$\piz:\spaces\to \spaces$ is a functor and the quotient maps $q_{X}:X\to \pi_0(X)$ are the components of a natural transformation $q:id\to \pi_0$.
\end{proposition}

Certainly, if $X$ is locally path connected, then $\pizx$ is a discrete space. More generally, $\pizx$ is discrete $\Leftrightarrow$ every path component of $X$ is open $\Leftrightarrow$ $X$ is \textit{semi-locally} $0$-\textit{connected} in the sense that for every point $x\in X$, there is a neighborhood $U$ of $x$ such that the inclusion $i:U\to X$ induces the constant map $i_0:\piz(U)\to \pizx$.

From the following basic examples, we begin to see that separation axioms are not passed to path-component spaces even if $X$ is a compact metric space.

\begin{example}\label{sinecurveexample}
 If $P_1=\{0\}\times [-1,1]$ and $P_2=\{(x,-\sin(\pi/x))|x\in (0,1]\}$ so that $T=P_1\cup P_2$ is the closed topologists sine curve, then $\pi_0(X)=\{P_1,P_2\}$ is the two-point Sierpinski space with topology $\{\emptyset, \{P_2\},X\}$. Thus $\pi_0(X)$ is $T_0$ but not $T_1$.
\end{example}

\begin{example}
Given any discrete space $V$, let $\Gamma$ be a maximal tree in the complete graph with vertex set $V$. We construct the space $X$ from $V$ as follows: for each edge $[x,y]$ of $\Gamma$, attach two copies of $T$ (as in the previous example) to $V$, the first by $(0,0)\sim x$ and $(1,0)\sim y$ and the second by $(0,0)\sim y$ and $(1,0)\sim x$. It is easy to see that the resulting space $X$ has indiscrete path-component space $\pi_0(X)=\{[v]|v\in V\}$ with the same cardinality as $V$.
\end{example}

\begin{example}
A compact metric space has cardinality at most that of the continuum and so its path-component space must be restricted in the same way. Consider the following path-connected subspaces of the xy-plane where $[a,b]$ denotes the line segment from $a$ to $b$.
\begin{itemize}
\item $P_1=\{(x,2-\sin(\pi/x))|0<x\leq 1\}\cup [(1,2),(0,-2)]\cup (\{0\}\times [-3,-1])$
\item $P_2=\{(-x,-2+\sin(\pi/x))|0<x\leq 1\}\cup [(-1,-2),(0,2)]\cup (\{0\}\times [1,3])$
\end{itemize}
Let $D=P_1\cup P_2$ (See Figure \ref{indiscrete}) and note that $\pi_0(D)=\{P_1,P_2\}$ is indiscrete since each path component is neither open nor closed. Now let $C\subset [0,1]$ be a Cantor set in the z-axis and take $X=(D\times C)\cup (\{0\}\times[1,3]\times [0,1])$. Including the rectangle $\{0\}\times[1,3]\times [0,1]$ has the effect of combining the path components $P_1\times \{c\}\in \pi_0(D\times C)$, $c\in C$ into a single path component. Thus $X\subset \mathbb{R}^3$ is a compact metric space such that $\pi_0(X)$ is uncountable and indiscrete. 
\end{example}
\begin{figure}[H]
\centering \includegraphics[height=2in]{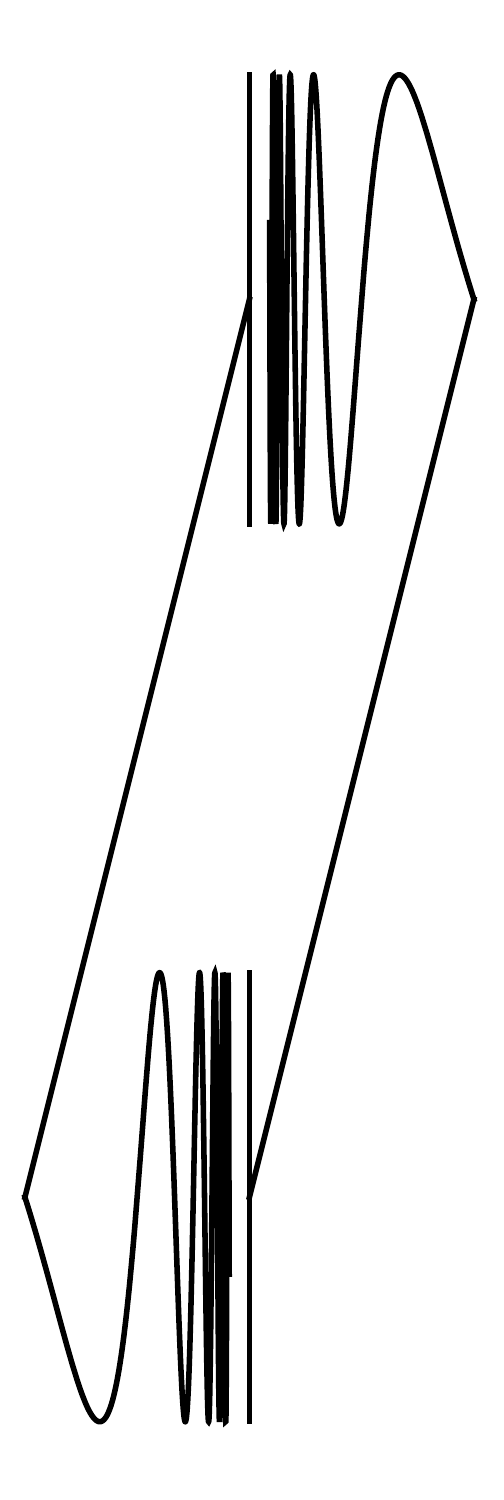}
\caption{\label{indiscrete}The space $D=P_1\cup P_2$.}
\end{figure}

At the other extreme, the natural quotient map $q_X:X \to\pizx$ is a homeomorphism if and only if $X$ is totally path disconnected, i.e. admits no non-constant paths. In general, $\pizx$ need not be totally path-disconnected. The existence of non-constant paths in $\pi_0(X)$ seems strange at first since it may permit cases where $\pi_0(\pi_0(X))$ is non-trivial. We make clear the relevance of this situation in our discussion on fundamental groups.

\begin{example}
Let $L$ be a linearly ordered space and $X=L\times [0,1]$ have the order topology induced by the lexicographical ordering. Certainly, the sets $\{\ell\}\times [0,1]$ are path-connected. If $\alpha:[0,1]\to X$ is a path from $(\ell_1,t_1)$ to $(\ell_2,t_2)$ where $\ell_1<\ell_2$, then by the intermediate value theorem the image of $\alpha$ contains all $(\ell,t)$ with $(\ell_1,t_1)\leq (\ell,t)\leq (\ell_2,t_2)$. By composing with the projection $L\times [0,1]\to L$, we see that $[\ell_1,\ell_2]$ is path connected and Hausdorff and therefore must be uncountable. Thus the image of $\alpha$ contains the uncountable family of disjoint open sets: $\{(\ell,t)|1/3<t<2/3\}$, $\ell\in(\ell_1,\ell_2)$, which is not possible.  It follows that the sets $\{\ell\}\times [0,1]$ are the path components of $X$ and that $\pi_0(X)\cong L$ where $q_X:X\to \pi_0(X)$ may be identified with the projection onto the first coordinate. Taking $L=[0,1]$ or $L=\mathbb{R}$ provides examples of non-metrizable $X$ for which $\pi_0(X)$ admits non-constant paths.
\end{example}

It is less clear how often $\pi_0(X)$ is totally path disconnected when $X$ is a metric space. A natural question one may then ask is: what spaces can be realized as path-component spaces? A beautiful and surprising result due to Douglas Harris shows that, in fact, every topological space is a path-component space in a functorial way.
\begin{theorem}[Harris \cite{Har80}] \label{essentiallysurjective}
For every space $X$, there is a paracompact Hausdorff space $H(X)$ and a natural homeomorphism $\piz(H(X))\cong X$. 
\end{theorem}

An analogue for metric spaces was identified by Banakh, Vovk, and W\' ojcik. Their most general result relating to path-component spaces is the following. Recall that \textit{premetric} on set $X$ is a function $d:X\times X\to [0,\infty)$ such that $d(x,x)=0$ for all $x\in X$. Any premetric generates a topology on $X$ by basic open sets $B_r(x)=\{y\in X|d(y,x)<r\}$ where $r>0$. A \textit{premetric space} is a topological space whose topology is generated by a premetric on $X$.

\begin{theorem}[Banakh, Vovk, W\' ojcik \cite{BVW}] Every premetric space $X$ is the path-component space of a complete metric space $\varoast(X)$ called the \textit{cobweb space} of $X$.
\end{theorem}
\begin{corollary}
Every metric space is the path-component space of a complete metric space.
\end{corollary}
\section{Realizing the unit interval as a path-component space}
The closed unit interval may be realized as the quotient of any infinite connected compact Hausdorff space $X$ using any Uryshon function $f:X\to [0,1]$ where $f(a)=0$ and $f(b)=1$ for $a\neq b$. We take a more specific approach here since our goal in this section is to realize the unit interval naturally as the path-component space of a space $K$ where we wish for $K$ to be as ``nice" as possible.

\begin{remark}\label{shrinkingremark}
We first recall an elementary fact from decomposition theory \cite{Daverman}. An onto map $f:X\to Y$ is said to be \textit{monotone} if the preimage $f^{-1}(y)$ is compact and connected for every $y\in Y$. It is well-known that if an onto map $q:[0,1]\to Z$ is monotone, then $Z\cong [0,1]$. 
\end{remark}
Let $C=\Big\{\sum_{i=1}^{\infty}\frac{x_i}{3^i}\Big| x_i\in \{0,2\} \Big\}$ be the standard middle third Cantor set in $[0,1]$. Let $\mathcal{C}$ be the set of open intervals which are the connected components of $[0,1]\backslash C$. For each $k\geq 1$, let $\mathcal{C}_{k}\subseteq \mathcal{C}$ be the set of connected components of $[0,1]\backslash C$ which are open intervals of length $\frac{1}{3^k}$.

Let $K_{odd}=\bigcup_{k\text{ odd}}\mathcal{C}_k$ and $K_{even}=\bigcup_{k\text{ even}}\mathcal{C}_k$. Define \[K=(C\times [0,1])\cup (K_{even}\times \{0\})\cup (K_{odd}\times \{1\})\] (See Figure \ref{thespacek}).
\begin{figure}[H]
\centering \includegraphics[height=1.8in]{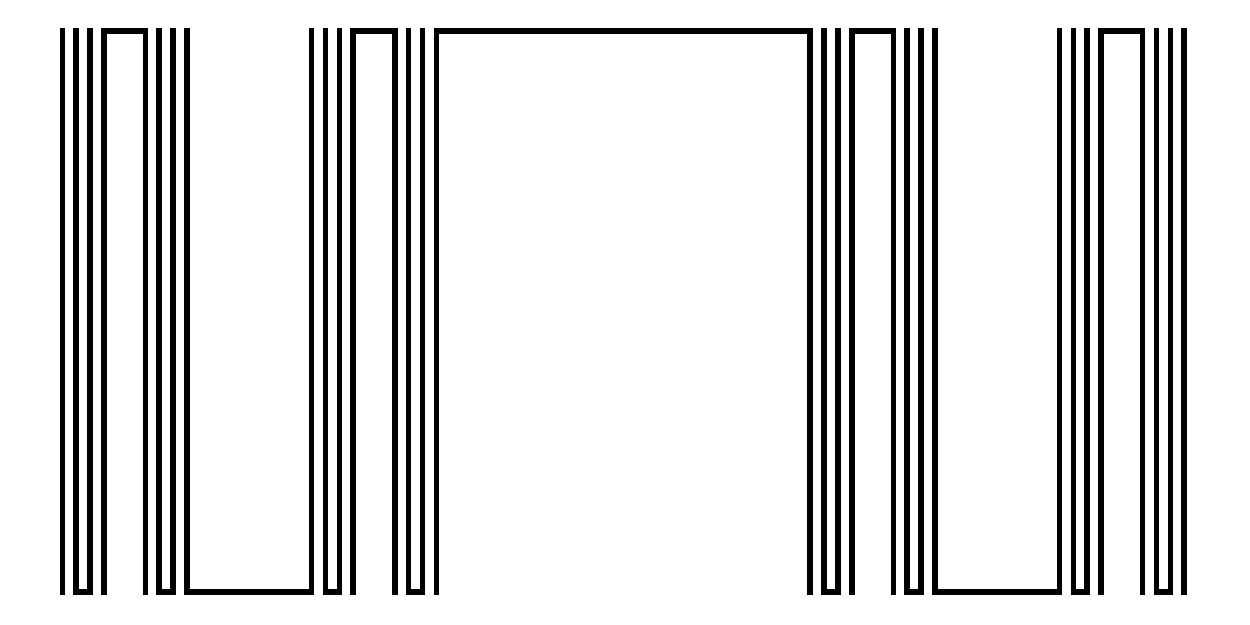}
\caption{\label{thespacek}The space $K$.}
\end{figure}
Alternatively, let $K_0=[0,1]^2$. For each $n\geq 1$ the set \[U_{n}=\bigcup_{\txt{\scriptsize$1\leq j\leq n$\\  \scriptsize $j$ odd}}\bigcup_{J\in \mathcal{C}_{j}}J\times [0,1)\cup \bigcup_{\txt{\scriptsize$1\leq j\leq n$\\  \scriptsize $j$ even}}\bigcup_{J\in \mathcal{C}_{j}}J\times (0,1]\] is open in $[0,1]^2$ and $K_n=[0,1]^2\backslash U_n$ is a 2-dimensional, contractible, compact CW-complex. Therefore, $K=\bigcap_{n\in\mathbb{N}}K_n$ is a compact metric space shape equivalent to a point.

\begin{definition}
We say a map $f:X\to Y$ is \textit{perfect} if $X$ is Hausdorff, $f$ is a closed map, and the point-preimage $f^{-1}(y)$ is compact for each $y\in Y$.
\end{definition}

\begin{theorem}\label{Kpcs}
$\pi_0(K)\cong [0,1]$ where the quotient map $q_K:K\to \pi_0(K)$ is a perfect map.
\end{theorem}
\begin{proof}
First, we identify the path components of $K$. Let $\pi:K\to [0,1]$ be the projection onto the first coordinate. Let $E\subseteq C$ be the set containing the endpoints of the intervals $I\in\mathcal{C}$ and set $D=C\backslash E$. For $c\in D$, $Y_c=\pi^{-1}(c)=\{c\}\times [0,1]$ is path-connected and for each $I\in\mathcal{C}$, $Y_I=\pi^{-1}\left(\overline{I}\right)\cong [0,1]$ is path-connected. For convenience, we say $I\in \mathcal{C}$ is of \textit{odd-type} if $I\subseteq K_{odd}$ and of \textit{even-type} if $I\subseteq K_{even}$. For any two distinct elements of $\{Y_c|c\in D\}\cup \{Y_I|I\in\mathcal{C}\}$, there are infinitely many intervals $I\in\mathcal{C}$ of odd-type and of even-type lying between them (as subsets of $[0,1]$).

If $A,B$ are connected sets in $[0,1]$ and $a<b$ for all $a\in A$ and $b\in B$, we write $A<B$. Suppose $\alpha:[0,1]\to K$ where $\alpha(0)$ and $\alpha(1)$ lie respectively within distinct elements $A$ and $B$ of $\{Y_c|c\in D\}\cup \{Y_I|I\in\mathcal{C}\}$. Without loss of generality, suppose $A<B$.

Inductively applying the intermediate value theorem, we may find an increasing sequence $0<s_1<t_1<s_2<t_2<...<1$ and sets $A<I_1<J_1<I_2<J_2<...<B$ where $I_n\in \mathcal{C}$ is of odd-type, $J_n\in\mathcal{C}$ of even-type,  $\pi\circ \alpha(s_n)=a_n$ is the midpoint of $I_n$, and $\pi\circ \alpha(t_n)=b_n$ is the midpoint of $J_n$. Note $\alpha(s_n)=(a_n,1)$ and $\alpha(t_n)=(b_n,0)$. If $x=\sup\{s_n|n\geq 1\}=\sup \{t_n|n\geq 1\}$, then by continuity of $\alpha$, we have $\alpha(x)=(\lim_{n\to\infty}a_n,1)$ in $[0,1]\times \{1\}$ and $\alpha(x)=(\lim_{n\to\infty}  b_n,0)$ in $[0,1]\times \{0\}$; a contradiction. We conclude that $\pi_0(K)=\{Y_c|c\in D\}\cup \{Y_I|I\in\mathcal{C}\}$ as a set.

Consider the map $f:[0,1]\to Z$ which, for each $I\in\mathcal{C}$, collapses $\overline{I}$ to a point. Certainly, $f$ is monotone so $Z\cong [0,1]$ by Remark \ref{shrinkingremark}. The fibers of $f\circ \pi$ are exactly the fibers of the quotient map $q_K:K\to \pi_0(K)$ so there is a unique continuous bijection $g:\pi_0(K)\to Z$ such that $g\circ q_K=f\circ \pi$. Since $K$ is compact, so is its quotient $\pi_0(K)$. Since $Z$ is Hausdorff, $g$ is a closed map. We conclude that $g$ is a homeomorphism and thus $\pi_0(K)\cong Z\cong [0,1]$.
\[\xymatrix{
K \ar[d]_{q_K} \ar[r]^-{\pi} & [0,1] \ar[d]^{f}\\
\pi_0(K) \ar@{->}[r]_-{g} & Z
}\]
Since $K$ is compact and $\pi_0(K)$ is Hausdorff, $q_K$ is a perfect map.
\end{proof}
\begin{example}
The space $K$ can also be used to construct reasonably nice spaces whose path-component spaces satisfy only weak separation axioms. The space $K'=(K\times [0,1])\backslash (\{1\}\times[0,1]\times (0,1))\subset \mathbb{R}^3$ is not compact but is a separable metric space. The path components of $K'$ are precisely the sets:
\begin{enumerate}
\item $C\times [0,1]$ where $C\subset [0,1)\times [0,1]$ is a path component of $K$,
\item $C_{0}=\{1\}\times [0,1]\times \{0\}$,
\item $C_{1}=\{1\}\times [0,1]\times \{1\}$.
\end{enumerate}
Based on Theorem \ref{Kpcs}, it follows that $\pi_0(K')\cong [0,1]\times \{0,1\}/\mathord{\sim}$ where $(t,0)\sim (t,1)$ if $t<1$. Hence, $\pi_0(K')$ is the $T_1$ but non-Hausdorff ``closed unit interval with two copies of $1$." 
\end{example}
\section{Realizing Tychonoff spaces as path-component spaces}
This section is dedicated to proving Theorem \ref{mainthm} as well as a functorial version of the construction.
\begin{proposition} \label{productsofpcs}
Let $\{X_{\lambda}\}$ be a family of spaces and $X=\prod_{\lambda} X_{\lambda}$. Let $q_{\lambda}:X_{\lambda}\to \pi_{0}(X_{\lambda})$ and $q_X:X\to \pizx$ be the canonical quotient maps and $\prod_{\lambda}q_{\lambda}:X\to \prod_{\lambda} \pi_{0}(X_{\lambda})$ be the product map. There is a natural continuous bijection $\phi:\pi_{0}\left(X\right)\to \prod_{\lambda}\pi_{0}(X_{\lambda})$ such that $\phi\circ q_X=\prod_{\lambda}q_{\lambda}$, which is a homeomorphism if and only if the product map $\prod_{\lambda}q_{\lambda}:\prod_{\lambda} X_{\lambda}\to \prod_{\lambda} \pi_{0}(X_{\lambda})$ is a quotient map.
\end{proposition}
\begin{proof}
The projections $r_{\lambda}:X\to X_{\lambda}$ induces maps $(r_{\lambda})_{0}:\pi_{0}\left(X\right)\to \pi_{0}(X_{\lambda})$ which in turn induce the natural map $\phi:\pi_{0}\left(X\right)\to \prod_{\lambda}\pi_{0}(X_{\lambda})$, $\phi([(x_\lambda)])=([x_{\lambda}])$. Clearly $\phi$ is surjective. If $[x_{\lambda}]=[y_{\lambda}]$ for each $\lambda$, then there is a path $\alpha_{\lambda}:[0,1]\to X_{\lambda}$ from $x_{\lambda}$ to $y_{\lambda}$. These maps induce a path $\alpha:[0,1]\to X$ such that $r_{\lambda}\circ \alpha=\alpha_{\lambda}$. Since $\alpha$ is a path from $(x_\lambda)$ to $(y_\lambda)$, we have $[(x_\lambda)]=[(y_\lambda)]$. Thus $\phi$ is injective.
\[\xymatrix{
& \prod_{\lambda} X_{\lambda} \ar[dl]_{q_X} \ar[dr]^{\prod_{\lambda} q_{\lambda}} \\
\pi_0\left(\prod_{\lambda} X_{\lambda}\right) \ar[rr]_{\phi} & & \prod_{\lambda}\pi_{0}(X_{\lambda})
}\]
The last fact follows directly from the fact that $q_X$ is a quotient map, $\phi$ is a bijection, and the commutativity of the triangle above.
\end{proof}
\begin{corollary}\label{compactpcs}
If, for all $\lambda$, $X_{\lambda}$ is compact and $\pi_0(X_{\lambda})$ is Hausdorff, then $\phi:\pi_{0}\left(\prod_{\lambda}X_{\lambda}\right)\to \prod_{\lambda}\pi_{0}(X_{\lambda})$ is a homeomorphism.
\end{corollary}
\begin{proof}
Given the assumptions, $\prod_{\lambda}X_{\lambda}$ is compact by the Tychonoff Theorem and $\prod_{\lambda}\pi_{0}(X_{\lambda})$ is Hausdorff. Thus $\prod_{\lambda} q_{\lambda}$ is a closed surjection and therefore quotient and we may apply the final statement of Proposition \ref{productsofpcs}.
\end{proof}
\begin{remark}\label{heredquotient}
Certainly every onto perfect map is quotient. Moreover, a perfect map $f:X\to Y$ has the property that if $S\subseteq Y$, then the restriction $f|_{f^{-1}(S)}:f^{-1}(S)\to S$ is a perfect map \cite[Proposition 3.7.6]{Eng89}.
\end{remark}
\begin{lemma}\label{basicprop}
Suppose $A$ is a compact Hausdorff space such that $\pi_0(A)$ is Hausdorff and $X\subseteq \pi_0(A)$. If $B=q_{A}^{-1}(X)\subseteq A$, then there is a canonical homeomorphism $\pi_0(B)\cong X$ and $q_B:B\to \pi_0(B)=X$ is a perfect map.
\end{lemma}
\begin{proof}
Since $A$ and $\pi_0(A)$ are compact Hausdorff, the quotient map $q_A:A\to \pi_0(A)$ is a perfect map. According to Remark \ref{heredquotient}, the restriction of $q_A$ to $B$ is a perfect map $q_{A}|_{B}:B\to X$. Since onto perfect maps are quotient, $q_{A}|_{B}$ is a quotient map which makes the same identifications as $q:B\to \pi_0(B)$. Thus there is a canonical homeomorphism $\pi_0(B)\cong X$.
\end{proof}
Recall that the \textit{weight} $w(X)$ of a topological space $X$ is the minimal cardinality of a basis generating the topology of $X$.\\\\
\noindent \textit{Proof of Theorem \ref{mainthm}.} Suppose $X$ is a Tychonoff space of weight $\mathbf{m}=w(X)$. Recall that for any cardinal $\mathbf{m}\geq \aleph_0$, the direct product $[0,1]^{\mathbf{m}}$ (of weight $\mathbf{m}$) is universal for all Tychonoff spaces $X$ of weight $\mathbf{m}$ \cite[2.3.23]{Eng89}, i.e. $X$ homeomorphically embeds as a subspace of $[0,1]^{\mathbf{m}}$. Hence, we may identify $X$ as a subspace of $[0,1]^{\mathbf{m}}$. Recalling the compact space $K\subset [0,1]^2$ constructed in the previous section, we may identify $\pi_0(K)=[0,1]$ by Theorem \ref{Kpcs}. By Corollary \ref{compactpcs}, the product map $Q=(q_{K})^{\mathbf{m}}:K^{\mathbf{m}}\to [0,1]^{\mathbf{m}}$ is a quotient map whose fibers are the path-components of $K^{\mathbf{m}}$ and the canonical bijection $\phi:\pi_0\left(K^{\mathbf{m}}\right)\to [0,1]^{\mathbf{m}}$ such that $\phi\circ q_{K^{\mathbf{m}}}=Q$ is a homeomorphism. Moreover, the Hausdorff spaces $K^{\mathbf{m}}$ and $[0,1]^{\mathbf{m}}$ are compact by the Tychonoff Theorem. Thus $Q$ is a closed map and hence a perfect map.
\[\xymatrix{
& K^{\mathbf{m}} \ar[dl]_{q_{K^{\mathbf{m}}}} \ar[dr]^{Q} \\
\pi_0\left(K^{\mathbf{m}}\right) \ar[rr]_{\phi} & & [0,1]^{\mathbf{m}}
}\]
Set $Y=Q^{-1}(X)\subseteq K^{\mathbf{m}}$. Note that since $K$ is second countable, $w(Y)\leq w(K^{\mathbf{m}})=\mathbf{m}=w(X)$. According to Remark \ref{heredquotient}, the restriction $Q|_{Y}:Y\to X$ of $Q$ is an onto perfect map. In general, If $f:Y\to X$ is an onto perfect map, then $w(X)\leq w(Y)$ \cite[3.7.19]{Eng89}. Hence $w(Y)=w(X)$. Since $\phi\circ q_{K^{\mathbf{m}}}=Q$, the quotient map $q_Y:Y\to \pi_0(Y)$ makes the same identifications as $Q|_{Y}$ and hence, $\phi$ restricts to a homeomorphism $\phi|_{\pi_0(Y)}:\pi_0(Y)\cong X$. Since $\phi|_{\pi_0(Y)}\circ q_Y=Q|_{Y}$ where $\phi|_{\pi_0(Y)}$ is a homeomorphism and $Q|_{Y}$ is perfect, $q_Y$ is also a perfect map. \hfill $\square$\\\\

The construction of $Y$ in the proof of Theorem \ref{mainthm} is not functorial since a choice of basis for the topology of $X$ is required for the embedding $X\subseteq [0,1]^{\mathbf{m}}$. We now show that if one is willing to give up the equality $w(Y)=w(X)$, the construction can be made functorial.
\begin{theorem}\label{functorthm}
Let $\mathbf{Tych}\subset \spaces$ be the full subcategory of Tychonoff spaces. Then there is a functor $\Phi:\mathbf{Tych}\to \mathbf{Tych}$ and a natural isomorphism $\pi_0\circ \Phi\cong Id_{\mathbf{Tych}}$ such that the natural maps $q_{\Phi(X)}:\Phi(X)\to \pi_0(\Phi(X))$ are onto perfect maps.
\end{theorem}
\begin{proof}
Let $X$ be a Tychonoff space and $C(X,I)$ denote the set of all continuous functions $X\to [0,1]$. Since this set of functions separates points and closed sets, the natural map $i_X:X\to [0,1]^{C(X,I)}$, $i(x)(f)=f(x)$ to the direct product is an embedding. For convenience, we identify $X=i(X)$. We follow the same line of argument used in the proof of Theorem \ref{mainthm} except that we replace $[0,1]^{\mathbf{m}}$ and $K^{\mathbf{m}}$ with the compact spaces $[0,1]^{C(X,I)}$ and $K^{C(X,I)}$ respectively. The product map $Q=(q_{K})^{C(X,I)}:K^{C(X,I)}\to [0,1]^{C(X,I)}$ is perfect and we define $\Phi(X)=Q^{-1}(X)$. It follows that $\pi_0(\Phi(X))\cong X$ where $q|_{\Phi(X)}:\Phi(X)\to \pi_0(\Phi(X))$ is a perfect map.

To construct $\Phi$ on morphisms, suppose $h:X\to Y$ is a map of Tychonoff spaces. Let $m:K^{C(X,I)}\to K^{C(Y,I)}$ and $M:[0,1]^{C(X,I)}\to [0,1]^{C(Y,I)}$ be the canonical induced maps. Set $A=K^{C(X,I))}$ and $B=K^{C(Y,I)}$ so we can simply write the quotient maps identifying path components as $q_A$ and $q_B$ respectively. The naturality of the embeddings $i_X$ and quotient maps $q_X$ indicates that the following diagram commutes.
\[\xymatrix{
K^{C(X,I)} \ar[d]_{q_A} \ar[r]^-{m} & K^{C(Y,I)} \ar[d]^{q_B}\\
[0,1]^{C(X,I)} \ar[r]_-{M} & [0,1]^{C(Y,I)}\\
X \ar[u]^{i_X} \ar[r]_-{f} & Y \ar[u]_{i_Y}
}\]
Since $M(i_X(X))=i_Y(f(X))\subseteq i_Y(Y)$, we have $q_B(m(\Phi(X)))=q_B(m(q_{A}^{-1}(i(X))))=M(i_X(X))=i_Y(f(X))\subseteq i_Y(Y)$ and thus $m(\Phi(X))\subseteq q_{B}^{-1}(i_Y(Y))=\Phi(Y)$. Thus we may define $\Phi(h):\Phi(X)\to \Phi(Y)$ to be the restriction of $m$ to $\Phi(X)$. The rest of the details of confirming functorality of $\Phi$ and naturality of the homeomorphisms $\pi_0(\Phi(X))\to X$ are straightforward.
\end{proof}

\section{Application to topologized fundamental groups}

For a path-connected space $X$ with basepoint $x_0\in X$, let $\Omega(X,x_0)$ denote the space of based maps $(S^1,b_0)\to (X,x_0)$ with the compact-open topology. The path-component space $\pi_1(X,x_0)=\pi_0(\Omega(X,x_0))$ is the fundamental group equipped with the natural quotient topology. It is known that $\pi_1(X,x_0)$ is a quasitopological group in the sense that inversion is continuous and the group operation is continuous in each variable. However, $\pi_1(X,x_0)$ can fail to be a topological group \cite{Br10.1,Fab10,Fab11}. A general study of fundamental groups with the quotient topology appears in \cite{BFqtop}.

The following construction is taken from the paper \cite{Br10.1}. For $X$ a compact metric space let $W(X)=X\times S^1/X\times \{b_0\}$ where the image of $X\times \{b_0\}$ is the basepoint $w_0$. Certainly $W(X)$ is also a compact metric space. Equivalently, if $X_+=X\cup\{\ast\}$ is $X$ with an added isolated basepoint, then $W(X)\cong \Sigma(X_+)$ is the reduced suspension of $X_+$. We think of $W(X)$ as a wedge of circles where the circles are topologically parameterized by the space $X$. In particular, if $X=\coprod_{j\in J}A_j$ is a topological sum of contractible spaces $A_j$, then $W(X)$ is homotopy equivalent to $\bigvee_{j\in J}S^1$.

We recall the notion of free topological group in the senses of Markov \cite{Markov} and Graev \cite{Graev}. Since the groups $\pi_1(X,x_0)$ need not be Hausdorff, we do not assume any separation axioms.

\begin{definition}
The \textit{free Markov topological group} $F_M(Y)$ on the space $Y$ is the topological group equipped with a map $\sigma:Y\to F_M(Y)$ such that every map $f:Y\to G$ to a topological group $G$ extends uniquely to a continuous homomorphism $\widetilde{f}:F_M(Y)\to G$ such that $\sigma \circ\widetilde{f}=f$.

The \textit{free Graev topological group} $F_G(Y,y)$ on the based space $(Y,y)$ is the topological group equipped with a map $\sigma_{\ast}:Y\to F_G(Y,y)$ taking $y$ to the identity element $e$ such that every map $f:(Y,y)\to (G,1_{G})$ to a topological group $G$ extends uniquely to a continuous homomorphism $\widetilde{f}:F_G(Y,y)\to G$ such that $\sigma_{\ast} \circ\widetilde{f}=f$.
\end{definition}
\begin{remark}\label{ftgremark}
By the Adjoint Functor Theorem, the free Markov (Graev) topological group exists for all (based) spaces $Y$. The underlying group of $F_M(Y)$ is the free group on the underlying set of $Y$ and the underlying group of $F_G(Y,y)$ is the free group on $Y\backslash \{y\}$. The groups $F_M(Y)$ and $F_G(Y,y_0)$ are Hausdorff if $Y$ is functionally Hausdorff\footnote{A space $X$ is \textit{functionally Hausdorff} if for any distinct points $a,b\in X$, there is a continuous function $f:X\to[0,1]$ such that $f(a)=0$ and $f(b)=1$.} \cite{Thomas}. The universal maps $\sigma$ and $\sigma_{\ast}$, which are the inclusion of generators, are closed embeddings if $Y$ is Tychonoff. The homomorphism $F_M(Y)\to F_G(Y,y)$ which factors the normal subgroup generated by $\{y\}$ is a topological quotient map. In general, the topologies of $F_M(Y)$ and $F_G(Y,y)$ are fairly difficult to characterize. We refer to \cite{AT08,Samuel,Sipacheva,Thomas} for more on the existence and structure of free topological groups.
\end{remark}

\begin{proposition}\label{pathconnftg}
If $Y$ is path-connected, then so is $F_G(Y,y)$.
\end{proposition}

\begin{proof}
The image of $\sigma_{\ast}:Y\to F_G(Y,y)$ is path-connected and algebraically generates the entire group. Since the path-component of the identity in a topological group is a subgroup, this must be the entire group $F_G(Y,y)$.
\end{proof}

A topological space $X$ is said to be a $k_{\omega}$\textit{-space} if $X$ is the inductive limit of compact subspaces $X_1\subset X_2\subset ...$, i.e. $X=\bigcup_{n\geq 1}X_n$ and $C\subseteq X$ is closed if and only if $C\cap X_n$ is closed in $X_n$ for all $n$.

\begin{theorem}[Mack, Morris, Ordman \cite{MMO}] \label{komegathm} If $X$ is a Hausdorff $k_{\omega}$-space and $x_0\in X$, then $F_M(X)$ and $F_G(X,x_0)$ are Hausdorff $k_{\omega}$-spaces.
\end{theorem}

The following Theorem follows from Corollaries 1.2 and 4.23 of \cite{Br10.1}.

\begin{theorem}\label{fmpione} 
If $X$ is a $k_{\omega}$-space, then $\pi_1(W(X),w_0)$ is naturally isomorphic to the free Markov topological group $F_M(\pi_0(X))$ on the path-component space $\pi_0(X)$.
\end{theorem}

By combining Theorem \ref{fmpione} with Theorem \ref{mainthm}, we realize all free topological groups on compact Hausdorff spaces as the fundamental group of a compact Hausdorff space. We remark that results in \cite{Br10.1} had no control over the compactness or metrizability of the construction since they depended on Harris' construction.

\begin{corollary}\label{unbasedrealization} \cite[Corollary 1.2]{Br10.1}
For every Tychonoff space $X$, there is a Tychonoff space $Y$ such that $w(Y)=w(X)$, $q:Y\to \pi_0(Y)=X$ is a perfect map, and $\pi_1(W(Y),w_0)\cong F_M(X)$.
\end{corollary}

\begin{remark}\label{specialcaseremark}
For the case $X=[0,1]^n$, we need not use the full power of Theorem \ref{mainthm}. In this case, we may replace $Y$ with $K^n$ since $\pi_0(K^n)\cong[0,1]^n$. Since $K^n\subseteq \mathbb{R}^{2n}$, the space $W(K^n)$ embeds as a subspace of $\mathbb{R}^{2n+1}$ when considered as a union of circles (indexed over $K^n$) with a single point in common.
\end{remark}

\begin{lemma}\label{quotientlemma}\cite[Lemma 4.4]{Brazasdiss}
If a 2-cell $e^2$ is attached to a path-connected space $Y$, then the inclusion $Y\to Y\cup e^2$ induces a homomorphism $\pi_1(Y,y_0)\to \pi_1(Y\cup e^2,y_0)$ which is a topological quotient map.
\end{lemma}

\begin{theorem}\label{graevrealization}
Suppose $X$ is a compact Hausdorff space which is well-pointed at $x_0\in X$, i.e. $\{x_0\}\to X$ is a cofibration. Then $\pi_1(\Sigma X,w_0)\cong F_G(\pi_0(X),[x_0])$.
\end{theorem}
\begin{proof}
Since $\{x_0\}\to X$ is a cofibration, $\Sigma X$ is homotopy equivalent to $W(X)\cup e^2$ where the 2-cell is attached along the loop $\ell:S^1\to W(X)$ where $\ell(t)$ is the image of $(x_0,t)$ in $W(X)$. Recall by Theorem \ref{fmpione} that $\pi_1(W(X),w_0)\cong F_M(\pi_0(X))$ and that $[x_0]\in \pi_0(X)$ denotes the path component of $x_0$ in $X$. According to Remark \ref{ftgremark}, there is an isomorphism of topological groups $F_G(\pi_0(X),[x_0])\cong F_M(\pi_0(X))/N$ where $N$ is the normal subgroup generated by $\{[x_0]\}$. The inclusion $W(X)\to W(X)\cup e^2$ induces a homomorphism $\pi_1(W(X),w_0)\to \pi_1(W(X)\cup e^2,w_0)$, which according to Lemma \ref{quotientlemma} is a topological quotient map whose kernel is precisely $N$. Thus $\pi_1(W(X)\cup e^2,w_0)\cong F_G(\pi_0(X),[x_0])$.
\end{proof}
\begin{corollary} \label{basedrealization}\cite[Corollary 1.2]{Br10.1}
For every compact Hausdorff (resp. metric) space $X$ and $x_0\in X$, there is a path-connected compact Hausdorff (resp. metric) space $Z$ and $w_0\in Z$ such that $\pi_1(Z,w_0)\cong F_G(X,x_0)$.
\end{corollary}

\begin{proof}
Using Theorem \ref{mainthm}, construct a compact Hausdorff space $Y$ such that $\pi_0(Y)\cong X$. Pick a point $y\in Y$ and attach a copy of the unit interval to $Y$ where $1\sim y$. Call the resulting space $Y^+$ and let $y_{0}$ be the image of $0$, which we take to be the basepoint. Now $(Y^+,y_0)$ is well-pointed, and $\pi_0(Y^+)\cong \pi_0(Y)=X$. Set $Z=\Sigma Y^+$ and apply Theorem \ref{graevrealization}.
\end{proof}

\begin{example}\label{finalexample}
As in Remark \ref{specialcaseremark}, we may perform a more direct construction in the case $X=[0,1]^n$ and $x_0=\mathbf{0}\in X$. Set $y_0=\mathbf{0}\in Y=K^n$ and recall $\pi_0(Y)=[0,1]^n$. Let $Y^+$ be the space obtained by attaching a copy of the unit interval to $Y$ where $1\sim y_0$. Now we have the desired isomorphisms $\pi_1(\Sigma Y^+,w_0)\cong F_G(\pi_0(Y^+),[x_0])\cong F_G(\pi_0(Y),[k_0])\cong F_G([0,1]^n,\overline{0})$. 

Since $Y^+$ embeds in $\mathbb{R}^{2n}$, $\Sigma Y^+$ may be embedded as a compact subspace of $\mathbb{R}^{2n+1}$. Moreover, since $K^n$ is the intersection of contractible polyhedra, so are both of the spaces $Y^+$ and $\Sigma Y^+$. Hence, $\Sigma Y^+$ is shape equivalent to a point. These examples are remarkable, because $\Sigma Y^+$ is a finite dimensional compact metric space and yet $\pi_1(\Sigma Y^+,w_0)$ is a \textit{path-connected} topological group.
\end{example}
The first shape homotopy group of a paracompact Hausdorff space is the inverse limit $\check{\pi}_1(X,x_0)=\varprojlim_{\mathscr{U}}\pi_1(|N(\mathscr{U})|,v_0)$ of fundamental groups of geometric realizations of nerves of open covers $\mathscr{U}$ of $X$. Since we only require general properties of this group, we refer the reader to \cite{MS82} for details of the construction. The nerve $|N(\mathscr{U})|$ is a simplicial complex and hence has discrete fundamental group \cite{CM}. Hence $\check{\pi}_1(X,x_0)$ is naturally topologized as an inverse limit of discrete groups. It is known that there is a canonical homomorphism $\psi:\pi_1(X,x_0)\to \check{\pi}_1(X,x_0)$ which is continuous with respect to the quotient topology on $\pi_1(X,x_0)$ \cite{BFqtop}. The easiest way to ensure that $\pi_1(X,x_0)$ is Hausdorff is to check that $\psi$ is injective (since any space which injects into a Hausdorff space is Hausdorff). For example, $\psi$ is known to be injective for all one-dimensional spaces \cite{EK98} and planar spaces \cite{FZ05}. In the case when $\psi$ is injective, one may conclude that $\pi_1(X,x_0)$ is functionally Hausdorff since it continuously injects into the functionally Hausdorff space $\check{\pi}_1(X,x_0)$. However, $\phi$ is not typically an embedding \cite{Fab05.3}, so even when it is injection one can not immediately conclude that $\pi_1(X,x_0)$ enjoys stronger separation axioms.\\\\

Finally, we complete the proof of Theorem \ref{mainthm2}. Recall that a topological space is \textit{cosmic} if it is the continuous image of a separable metric space.\\\\

\noindent\textit{Proof of Theorem \ref{mainthm2}.} Using the case $n=1$ from Example \ref{finalexample}, we may construct a compact suspension space $X\subset \mathbb{R}^3$ such that $\pi_1(X,x_0)\cong F_G([0,1],0)$. Since $[0,1]$ is functionally Hausdorff, $F_G([0,1],0)$ is Hausdorff. In fact, it is a result of M.~Zarichnyi \cite{Zarichnyi} that $F_G([0,1],0)$ is homeomorphic to $\mathbb{R}^{\infty}$ topologized as the direct limit of finite Euclidean space $\mathbb{R}^n$, which is a locally convex linear topological space. Since $[0,1]$ is a $k_{\omega}$-space, it follows from Theorem \ref{komegathm} that $F_G([0,1],0)$ is a $k_{\omega}$-space. Lastly, since $X$ is a compact metric space, $\Omega(X,x_0)$ is a separable metric space \cite[4.2.17 and 4.2.18]{Eng89} and it follows that the quotient $\pi_1(X,x_0)$ is cosmic. Every Hausdorff cosmic topological group is regular and Lindel\"of and is therefore $T_4$.

Finally, recall that the first shape homotopy group $\check{\pi}_1(X,x_0)$ is topologized as an inverse limit of discrete groups and is therefore totally path disconnected. But $\pi_1(X,x_0)$ is path-connected by Proposition \ref{pathconnftg}. Hence, the canonical continuous homomorphism $\psi:\pi_1(X,x_0)\to \check{\pi}_1(X,x_0)$ must be the trivial homomorphism. \hfill $\square$


\begin{thebibliography}{99}
\bibitem{AT08} A.~Arhangel'skii, M.~Tkachenko, \emph{Topological Groups and Related Structures}, Series in Pure and Applied Mathematics, Atlantis Studies in Mathematics, 2008.

\bibitem{BVW} T.~Banakh, M.~Vovk, M.R.~W\' ojcik, \emph{Connected economically metrizable spaces}, Fund. Math. 212 (2011), 145--173.


\bibitem{Brazasdiss} J.~Brazas, \textit{Homotopy Mapping Spaces}, Ph.D. Dissertation, University of New Hampshire, 2011.

\bibitem{Br10.1}
J.~Brazas, \emph{The topological fundamental group and free topological groups},
  Topology Appl. 158 (2011) 779--802.

\bibitem{Braztopgrp}
J. Brazas, \emph{The fundamental group as a topological group}, Topology Appl. 160 (2013) 170--188.

\bibitem{Br12NS}
J.~Brazas, \emph{Open subgroups of free topological groups}, Fundamenta Mathematicae 226 (2014) 17--40.

\bibitem{BFqtop} J.~Brazas, P.~Fabel, \emph{On fundamental groups with the quotient topology}, J. Homotopy and Related Structures 10 (2015) 71--91.

\bibitem{CM} J.~Calcut, J.~McCarthy, \emph{Discreteness and homogeneity of the topological fundamental group}, Topology Proc. 34 (2009) 339--349.

%
\bibitem{Daverman} R.J.~Daverman, \textit{Decompositions of manifolds}, Pure and Applied Mathematics, vol. 124, Academic Press Inc., Orlando, FL, 1986.

\bibitem{EK98} K.~Eda, K.~Kawamura, \emph{The fundamental groups of one-dimensional spaces}, Topology Appl. 87 (1998) 163--172.

\bibitem{Eng89} R.~Engelking, \emph{General topology}, Heldermann Verlag Berlin, 1989.
%
\bibitem{Fab05.3} P.~Fabel, \emph{The topological hawaiian earring group does not embed in the inverse limit of free groups}, Algebraic \& Geometric Topology 5 (2005) 1585--1587.

\bibitem{Fab10}
P.~Fabel, \emph{Multiplication is discontinuous in the hawaiian earring group}, Bull. Polish Acad. Sci. Math. 59 (2011) 77--83

\bibitem{Fab11}
P.~Fabel, \emph{Compactly generated quasitopological homotopy groups with
  discontinuous multiplication}, To appear in Topology Proc.

\bibitem{FRVZ11} H.~Fischer, D.~Repov\v{s}, Z.~Virk, A.~Zastrow, \emph{On
semilocally simply connected spaces}, Topology Appl. \textbf{158} (2011),
no.~3, 397--408.

\bibitem{FZ05} H.~Fischer, A.~Zastrow, \emph{The fundamental groups of
subsets of closed surfaces inject into their first shape groups}, Algebraic
and Geometric Topology 5 (2005) 1655--1676.


\bibitem{FZ07} H.~Fischer, A.~Zastrow, \emph{Generalized universal
covering spaces and the shape group}, Fund. Math. \textbf{197} (2007),
167--196.


%
\bibitem{Graev} M.I.~Graev, \emph{Free topological groups}. Amer. Math. Soc. Transl. 8 (1962) 305--365.

\bibitem{Har80}
D.~Harris, \emph{Every space is a path component space}, Pacific J. Math. 91 (1980) 95--104.


\bibitem{MMO} J.~Mack, S.A.~ Morris, E.T.~Ordman, \emph{Free topological groups and projective dimension of locally compact abelian subgroups}. Proc. Amer. Math. Soc. 40 (1973) 303--308.

\bibitem{MS82} S.~Marde\v{s}i\' c, J.~Segal, \emph{Shape theory}, North-Holland Publishing Company, 1982.

\bibitem{Markov} A.A.~Markov, \emph{On free topological groups}. Izv. Akad. Nauk. SSSR Ser. Mat. 9 (1945) 3-64 (in Russian); English Transl.: Amer. Math. Soc. Transl. 30 (1950) 11-88; Reprint: Amer. Math. Soc. Transl. 8 (1) (1962) 195--272.

\bibitem{Michael} E.~Michael, \textit{Bi-quotient maps and Cartesian products of quotient maps}, Ann. Inst. Fourier, Grenoble. 18, no. 2 (1968) 287--302.

\bibitem{Samuel} P.~Samuel, \emph{On univeral mappings and free topological groups}. Bull. Am. Math. Soc. 54 (1948) 591--598.

\bibitem{Sipacheva} O.V.~Sipacheva, \emph{The Topology of Free Topological Groups}. J. Math. Sci. Vol. 131, No. 4, (2005), 5765–-5838.


\bibitem{Thomas} B.V.S.~Thomas, \emph{Free topological groups}. General Topology and its Appl. 4 (1974) 51--72.

%

\bibitem{Zarichnyi} M.M.~Zarichnyi, \textit{Free topological groups of absolute neighbourhood retracts and infinite-dimensional manifolds}, Soviet Math. Dokl. 26 (1982) 367-–371.

\end{thebibliography}
\end{document}